\documentclass{amsart}

\usepackage{graphicx,amssymb,mathrsfs,amsmath,color,fancyhdr}
\usepackage[all]{xy}
\newtheorem{theorem}{Theorem}[section]
\newtheorem{lemma}[theorem]{Lemma}

\theoremstyle{definition}

\newtheorem{conjecture}[theorem]{Conjecture}

\newtheorem{corollary}[theorem]{Corollary}

\theoremstyle{remark}
\newtheorem{remark}[theorem]{Remark}

\numberwithin{equation}{section}

\newcommand{\bdot}{\boldsymbol{\cdot}}

\begin{document}

\title[On generalized Erd\H{o}s-Ginzburg-Ziv constants of $C_n^r$]
{On generalized Erd\H{o}s-Ginzburg-Ziv constants of $C_n^r$}

\begin{abstract}
Let $G$ be an additive finite abelian group with exponent $\exp(G)=m$. For any positive integer $k$, the $k$-th generalized Erd\H{o}s-Ginzburg-Ziv constant $\mathsf s_{km}(G)$ is defined as the smallest positive integer $t$ such that every sequence $S$ in $G$ of length at least $t$ has a zero-sum subsequence of length $km$. It is easy to see that $\mathsf s_{kn}(C_n^r)\ge(k+r)n-r$ where $n,r\in\mathbb N$. Kubertin conjectured that the equality holds for any $k\ge r$.
In this paper, we mainly prove the following results:
\begin{enumerate}
\item For every positive integer $k\ge 6$, we have
$$\mathsf s_{kn}(C_n^3)=(k+3)n+O(\frac{n}{\ln n}).$$
\item For every positive integer $k\ge 18$, we have
$$\mathsf s_{kn}(C_n^4)=(k+4)n+O(\frac{n}{\ln n}).$$
\item For $n\in \mathbb N$, assume that the largest prime power divisor of $n$ is $p^a$ for some $a\in\mathbb N$. For any fixed $r\ge 5$, if $p^t\ge r$ for some $t\in\mathbb N$, then for any $k\in\mathbb N$ we have
$$\mathsf s_{kp^tn}(C_n^r)\le(kp^t+r)n+c_r\frac{n}{\ln n},$$
where $c_r$ is a constant depends on $r$.
\end{enumerate}
Note that the main terms in our results are consistent with the conjectural values proposed by Kubertin. 
\end{abstract}

\author{Dongchun Han}
\address{Department of Mathematics, Southwest Jiaotong University, Chengdu 610000, P.R. China}
\email{han-qingfeng@163.com}
\author{Hanbin Zhang}
\address{Academy of Mathematics and Systems Science, Chinese Academy of Sciences, Beijing
100190, P.R. China}
\email{zhanghanbin@amss.ac.cn}

\keywords{}
\maketitle

\section{Introduction}

Let $G$ be an additive finite abelian group with exponent $\exp(G)=m$. Let $S=g_1\bdot\ldots\bdot g_k$ be a sequence over $G$ (repetition is allowed), where $g_i\in G$ for $1\le i\le k$, $k$ is called the length of the sequence $S$. We call $S$ a zero-sum sequence if $\sum^k_{i=1}g_i=0$. The classical direct zero-sum problem studies conditions (mainly refer to lengths) which ensure that given sequences have non-empty zero-sum subsequences with prescribed properties (also mainly refer to lengths). For example, the Davenport constant, denoted by $\mathsf D(G)$, is the smallest positive integer $t$ such that every sequence $S$ over $G$ of length at least $t$ has a nonempty zero-sum subsequence. It is easy to prove that $\mathsf D(C_n)=n$, where $C_n$ is the cyclic group of order $n$. For any positive integer $k$, the $k$-th generalized Erd\H{o}s-Ginzburg-Ziv constant $\mathsf s_{km}(G)$ is defined as the smallest positive integer $t$ such that every sequence $S$ over $G$ of length at least $t$ has a zero-sum subsequence of length $km$. In particular, for $k=1$, $\mathsf s_m(G)$ is called the Erd\H{o}s-Ginzburg-Ziv constant, which is a classical invariant in combinatorial number theory. In 1961, Erd\H{o}s, Ginzburg and Ziv \cite{EGZ} proved that $\mathsf s_n(C_n)=2n-1$ which is usually regarded as a starting point of zero-sum theory (see \cite{ADZ} for other different proofs of this result). We refer to \cite{GG} for a survey of zero-sum problems. In this paper, we will focus on $\mathsf s_{km}(G)$.

Let $G=C_n^r=\langle e_1\rangle\oplus\cdots\oplus\langle e_r\rangle$. Assume that $T$ consists of $n-1$ copies of $e_i$ for $1\le i\le r$. Let $S$ consist of $kn-1$ copies of $0$ and $T$, then it is easy to show that $S$ is a sequence over $C_n^r$ of length $(k+r)n-r-1$ and $S$ contains no zero-sum subsequences of length $kn$. Consequently we have
\begin{equation}\label{eq1.1}
\mathsf s_{kn}(C_n^r)\ge(k+r)n-r.
\end{equation}
For general finite abelian group $G$ with $\exp(G)=m$, similar construction can be used to show that $\mathsf s_{km}(G)\ge km+\mathsf D(G)-1$ holds for $k\ge1$. In 1996, Gao \cite{Gao3} proved that $\mathsf s_{km}(G)= km+\mathsf D(G)-1$, provided that $km\ge|G|$. In \cite{GaoThang}, Gao and Thangadurai proved that if $km<\mathsf D(G)$, then $\mathsf s_{km}(G)> km+\mathsf D(G)-1$. Define $l(G)$ as the smallest integer $t$ such that $\mathsf s_{km}(G)=km+\mathsf D(G)-1$ holds for every $k\ge t$. From the above we know that
$$\frac{\mathsf{D}(G)}{m}\le l(G)\le \frac{|G|}{m}.$$
Recently, Gao, Han, Peng and Sun conjectured (\cite{GHPS}, Conjecture 4.7) that $$l(G)=\lceil\frac{\mathsf D(G)}{m}\rceil.$$
Clearly we have $l(C_n)=1$ by the Erd\H{o}s-Ginzburg-Ziv theorem. For finite abelian groups $G$ of rank two, $l(G)=2$ (see \cite{GHPS}). Let $p$ be a prime and $q$ a power of $p$, the above conjecture was verified for $C_q^r$ where $1\le r\le 4$ (also more generally for abelian $p$-group $G$ with $\mathsf D(G)\le 4m$) except for some cases when $p$ is rather small, see \cite{GaoThang, HZ, K}. For the studies of $l(G)$ for the general cases, we refer to \cite{GHPS, He, K}.

Recall (\ref{eq1.1}) that $\mathsf s_{kn}(C_n^r)\ge(k+r)n-r$, in \cite{K}, Kubertin conjectured that the equality actually holds for any $k\ge r$.
\begin{conjecture}\label{conj1}
For any positive integers $k,n$ with $k\ge r$, we have
$$\mathsf s_{kn}(C_n^r)=(k+r)n-r.$$
\end{conjecture}

According to the results in \cite{GaoThang, HZ, K}, Conjecture \ref{conj1} has been verified for $r\le 4$ except for some cases when $p$ is rather small ($p\le 3$). Recently, Sidorenko \cite{S1,S2} verified Conjecture \ref{conj1} for $C_2^r$. He \cite{S1} also applied his results to prove new bounds for the codegree Tur\'{a}n density of complete $r$-graphs. Moreover, he \cite{S2} established connections between $\mathsf s_{2k}(C_2^r)$ and linear binary codes. Actually, he showed that the problem of determining $\mathsf s_{2k}(C_2^r)$ is essentially equivalent to finding the lowest redundancy of a linear binary code of given length which does not contain words of Hamming weight $2k$.

Towards Conjecture \ref{conj1}, Kubertin \cite{K} proved that
$$\mathsf s_{kq}(C_q^r)\le (k+\frac{3}{8}r^2+\frac{3}{2}r-\frac{3}{8})q-r,$$
where $p>\min\{2k,2r\}$ is a prime and $q$ is a power of $p$.
By extending the method of Kubertin, He \cite{He} improved the above upper bound and obtained that $$\mathsf s_{kq}(C_q^r)\le(k+5r-2)q-3r$$
when $2p \ge 7r-3$ and $k\ge r$. He also proved that $\mathsf s_{kn}(C_n^r)\le 6kn$ for $n$ with
large prime factors and $k$ sufficiently large. More precisely, he showed that for $r,l>0$, $n=p_1^{\alpha_1}\cdots p_l^{\alpha_l}$ with distinct prime factors $p_1,\ldots,p_l\ge\frac{7}{2}r-3$ and $k=a_1\cdots a_l$ a product of positive integers $a_1,\ldots,a_l\ge r$, $\mathsf s_{kn}(C_n^r)\le 6kn$. We also refer to \cite{BGH,Gen} for some recent results on the lower bound of $\mathsf s_{kn}(C_n^r)$ when $k$ is much smaller than the rank $r$, note that in this case $\mathsf s_{kn}(C_n^r)>(k+r)n-r$ (see \cite{GaoThang}).

For $n\in\mathbb N$, let
$$\mathsf M(n)=\max\{p^k|\text{ with }p^k|n\text{ where $p$ is a prime and k}\in \mathbb N\},$$
i.e., the largest prime power divisor of $n$. For convenience, let $\mathsf M(1)=1$. For any $n,r\in \mathbb N$, we define
$$\mathsf p(n,r)=\min\{p^t|\text{ }\mathsf M(n)=p^a\text{ and }p^t\ge r\}.$$

In this paper, we focus on the Conjecture \ref{conj1} and prove the following results.
\begin{theorem}\label{theorem1}
Let $k\in \mathbb N$. We have
\begin{enumerate}
\item For every $k\ge 6$,
$$\mathsf s_{kn}(C_n^3)=(k+3)n+O(\frac{n}{\ln n});$$
\item For every $k\ge 18$,
$$\mathsf s_{kn}(C_n^4)=(k+4)n+O(\frac{n}{\ln n});$$
\item For every $k\in \mathbb N$ and fixed $r\ge 5$,
$$\mathsf s_{k\mathsf p(n,r)n}(C_n^r)=(k\mathsf p(n,r)+r)n+O_r(\frac{n}{\ln n}),$$
where $O_r$ depends on $r$.
\end{enumerate}
\end{theorem}

Note that the main terms in Theorem \ref{theorem1} are consistent with the conjectural values in Conjecture \ref{conj1}. Moreover, the error term can be improved in some cases. By some further studies of $\mathsf M(n)$, roughly speaking, for any real number $A\ge 1$, we can improve the order of the error term from $\frac{n}{\ln n}$ to $\frac{n}{(\ln n)^A}$ for almost every $n\ge 1$. Furthermore, when the number of distinct prime divisors of $n$ is a given integer $m$, we can even improve the order of the error term to $n^{1-\frac{1}{m}}$.

The following sections are organized as follows. In Section 2, we shall introduce some notations and preliminary results. In Section 3, we will prove our main results. In Section 4, we will provide further studies on $\mathsf M(n)$ and then apply these results to improve our main results.

\section{Preliminaries}

This section will provide more rigorous definitions and notations. We also introduce some preliminary results that will be used repeatedly below.

Let $\mathbb{N}$ denote the set of positive integers, $\mathbb{N}_0=\mathbb{N}\cup\{0\}$ and $\mathbb R$ the field of real numbers. Let $f$ and $g$ be real valued functions, both defined on $\mathbb N$, such that $g(x)$ is strictly positive for all large enough values of $x$. Then we denote
$f(x)=O(g(x))$ if and only if there exists a positive real number $M$ and a positive integer $x_0$ such that
$$|f(x)|\le M|g(x)|\qquad {\text{ for all }}x\geq x_{0}.$$
We also use the notation $O_r$ (resp. $O_{A,\epsilon}$) which means that the above $M$ depends on $r$ (resp. $A$ and $\epsilon$), where $r\in\mathbb N_0$, $A,\epsilon\in \mathbb R$.
Similarly, we denote $f(x)=o(g(x))$ if and only if for every positive constant $\varepsilon$, there exists a positive integer $x_0$ such that
$$|f(x)|\le \varepsilon g(x)\qquad {\text{for all }}x\geq x_0.$$

Let $G$ be an additive finite abelian group. By the fundamental theorem of finite abelian groups we have
$$G\cong C_{n_1}\oplus\cdots\oplus C_{n_r}$$
where $r=\mathsf r(G)\in \mathbb{N}_0$ is the rank of $G$, $n_1|\cdots|n_r\in\mathbb{N}$ are positive integers. Moreover, $n_1,\ldots,n_r$ are uniquely determined by $G$, and $n_r=\exp(G)$ is called the $exponent$ of $G$.

We define a $sequence$ over $G$ to be an element of the free abelian monoid $\big(\mathcal F(G),\bdot\big)$, see Chapter 5 of \cite{GH} for detailed explanation. Our notations of sequences follow the notations in the paper \cite{GeG}. In particular, in order to avoid confusion between exponentiation of the group operation in $G$ and exponentiation of the sequence operation $\bdot$ in $\mathcal F (G)$, we define:
\[
g^{[k]}=\underset{k}{\underbrace{g\bdot\ldots\bdot g}}\in \mathcal F (G)\quad \text{and} \quad T^{[k]}=\underset{k}{\underbrace{T\bdot\ldots\bdot T}}\in \mathcal F (G) \,,
\]
for $g \in G$, \ $T\in \mathcal F (G)$ and $k \in \mathbb N_0$.

We write a sequence $S$ in the form
$$S=\prod_{g\in G}g^{\textsf{v}_g(S)}\text{ with }\textsf{v}_g(S)\in \mathbb{N}_0\text{ for all }g\in G.$$
We call
\begin{itemize}
\item $\textsf{v}_g(S)$ the $multiplicity$ of $g$ in $S$,

\item $|S|=l=\sum_{g\in G}\textsf{v}_g(S)\in \mathbb{N}_0$ the $length$ of $S$,

\item $T=\prod_{g\in G}g^{\textsf{v}_g(T)}$ a $subsequence$ of $S$ if $\textsf{v}_g(T)\le \textsf{v}_g(S)$ for all $g\in G$, and denote by $T|S$,

\item $\sigma(S)=\sum\limits_{i=1}\limits^{l}g_i=\sum_{g\in G}\textsf{v}_g(S)g\in G$ the $sum$ of $S$,

\item $S$ a $zero$-$sum$ $sequence$ if $\sigma(S)=0$,

\item $S$ a $zero$-$sum$ $free$ $sequence$ if $\sigma(T)\neq0$ for every $T|S$,

\item $S$ a $short$ $zero$-$sum$ $sequence$ if it is a zero-sum sequence of length $|S|\in[1,\text{exp}(G)]$.
\end{itemize}

Using these concepts, we can define

\begin{itemize}
\item $\mathsf D(G)$ as the smallest integer $l\in \mathbb{N}$ such that every sequence $S$ over $G$ of length $|S|\geq l$ has a non-empty zero-sum subsequence. We call $\mathsf D(G)$ the $Davenport$ $constant$ of $G$.

\item $\mathsf s_{k\exp(G)}(G)$ as the smallest integer $l\in \mathbb{N}$ such that every sequence $S$ over $G$ of length $|S|\geq l$ has a non-empty zero-sum subsequence $T$ of length $|T|=k\exp(G)$, where $k\in\mathbb N$. We call $\mathsf s(G):=\mathsf s_{\exp(G)}(G)$ the Erd\H{o}s-Ginzburg-Ziv constant and $\mathsf s_{k\exp(G)}(G)$ the $k$-th generalized Erd\H{o}s-Ginzburg-Ziv constant.

\end{itemize}

\begin{lemma} \label{lemma1.1}{\rm{(\cite{GH}, Theorem 5.5.9)}}
Let $G$ be a finite abelian $p$-group and $G=C_{p^{n_{1}}}\oplus\cdots\oplus C_{p^{n_{r}}}$, then
$$\mathsf D(G)=\sum\limits_{i=1}\limits^{r}(p^{n_i}-1)+1.$$
\end{lemma}

\begin{lemma}\label{lower}
Let $G$ be a finite abelian group with $\exp(G)=m$, then
$$\mathsf s_{km}(G)\ge km+\mathsf D(G)-1$$
holds for every $k\ge1$.
\end{lemma}
\begin{proof}
By the definition of $\mathsf D(G)$, there exists a zero-sum free sequence $T$ of length $|T|=\mathsf D(G)-1$. Let $S=T\bdot 0^{[km-1]}$. It is easy to know that $S$ is a sequence over $G$ of length $|S|=km+\mathsf D(G)-2$ and $S$ contains no zero-sum subsequence of length $km$. This completes the proof.
\end{proof}

\begin{lemma} \label{lemma1.2}{\rm{(\cite{Gao1}, Theorem 3.2)}}
Let $G$ be a finite abelian $p$-group and $\exp(G)=p^{n_r}$. If $p^{m+n_r}\ge\mathsf D(G)$ for some $m\in\mathbb N$, then
$$\mathsf s_{kp^mp^{n_r}}(G)=k\cdot p^{m+n_r}+\mathsf D(G)-1,$$
holds for any $k\in\mathbb N$.
\end{lemma}

The following classical result of Alon and Dubiner is crucial in our proof.

\begin{lemma}[\cite{AD}, Theorem 1.1]\label{lemma2.1}
There exists an absolute constant $c>0$ such that
$$\mathsf s(C_n^r)\le (cr\log_2r)^rn.$$
\end{lemma}
Although the precise values of $\mathsf s(G)$ for general $C_n^r$ are not known, some cases (when $n$ is a power of a small prime) have been determined. We list some of these results which are very useful in our proof.
\begin{lemma} \label{lemma1.3}
Let $n\in\mathbb N$.
\begin{enumerate}
\item $\mathsf s(C_{2^n}^3)=8\cdot2^n-7$;
\item $\mathsf s(C_{3^n}^3)=9\cdot3^n-8$;
\item $\mathsf s(C_{2^n}^4)=16\cdot2^n-15$;
\item $\mathsf s(C_{3^n}^4)=20\cdot3^n-19$.
\end{enumerate}
\end{lemma}
\begin{proof}
(1) See \cite{EE}, Corollary 4.4. (2) See \cite{GHST}, Theorem 1.7. (3) See \cite{EE}, Corollary 4.4. (4) See \cite{EE}, Theorem 1.3, 1.4 and Section 5.
\end{proof}

In the rest of this section, we provide some results about $\mathsf M(n)$ which are useful in this paper. Recall that, for any $n\in \mathbb N$, let
$$\mathsf M(n)=\max\{p^k\text{ }|\text{ with }p^k|n\text{ where $p$ is a prime and k}\in \mathbb N\}$$ be the largest prime power divisor of $n$. For convenience, let $\mathsf M(1)=1$. For example, we have $\mathsf M(20)=\mathsf M(2^25)=5$, $\mathsf M(40)=\mathsf M(2^35)=2^3$ and $\mathsf M(200)=\mathsf M(2^35^2)=5^2$. Unlike the widely studied largest prime divisor function
$$\mathsf P(n)=\max\{p\text{ }|\text{ with }p|n\text{ and $p$ is a prime}\},$$ as far as we know, $\mathsf M(n)$ has not received much attention. As $\mathsf M(p)=p$ where $p$ is a prime, certainly we have $\limsup\limits_{n\rightarrow\infty}\frac{\mathsf M(n)}{n}=1$. It is known and easy to prove that
\begin{equation}\label{eq2}
\liminf\limits_{n\rightarrow\infty}\frac{\mathsf M(n)}{\ln n}=1,
\end{equation}
consequently
\begin{equation}\label{eq3}
\lim\limits_{n\rightarrow\infty}\mathsf M(n)=\infty.
\end{equation}
Recently, Girard \cite{Gi} used (\ref{eq3}) to show that $\mathsf D(C_n^r)=rn+o(n)$, which is an important result in zero-sum theorey and also can be regarded as an example of application of $\mathsf M(n)$. In this paper, we will continue to employ the estimates of $\mathsf M(n)$ to the zero-sum problems. Although the proof of (\ref{eq2}) is simple and elementary, it is hard to find this result in literatures or standard textbooks. So we decide to provide a proof here for the convenience of the reader.

Let
$$\pi(x)=\#\{p\text{ $|$ }p\le x\}$$
be the prime-counting function and
$$\vartheta(x)=\sum_{p\le x}\ln p$$ the Chebyshev $\vartheta$ function. The result in the following lemma is very classical and can be easily found in \cite{T}.

\begin{lemma}\label{lemma3.1}For any $x\ge 2$, we have
$$\pi(x)\le 2\frac{x}{\ln x}.$$
\end{lemma}

\begin{proof}
This result is an easy consequence of Theorem 3, Page 11 in \cite{T}.
\end{proof}

\begin{lemma}\label{lemma3.2}
For any $n\in \mathbb N$, we have
$$\mathsf M(n)\ge \frac{1}{2}\ln n.$$
\end{lemma}

\begin{proof}
When $n=p^m$ is a prime power, the result is obvious. If $n$ is not a prime power, we may assume that
$$n=q_1^{r_1}\cdots q_k^{r_k}p^m,$$
where $q_1<\cdots< q_k$ and $p$ are distinct prime numbers, $r_1,\ldots,r_k,m\in\mathbb N$ with $\mathsf M(n)=p^m$. By the definition of $\mathsf M(n)$, clearly we have $n\le p^mp^{km}$. Moreover we have $k< \pi(p^m)$. For otherwise if $k\ge\pi(p^m)$, then we have $$k=\pi(p_k)\ge\pi(p^m)\ge\pi(p)$$
and consequently $p_k\ge p$. As $q_1<\cdots< q_k$ and $p$ are distinct prime numbers, we have $q_k>p_k$. Therefore,
$$\pi(q_k)>\pi(p_k)=k\ge\pi(p^m)$$
and consequently $q_k>p^m$. By the definition of $\mathsf M(n)$, we have $\mathsf M(n)\ge q_k^{r_k}>p^m$, but this contradicts $\mathsf M(n)=p^m$. Therefore $n\le p^{m\pi(p^m)}$, and by Lemma \ref{lemma3.1} we have
$$\frac{\mathsf M(n)}{\ln n}=\frac{p^m}{\ln n}\ge \frac{p^m}{\ln p^{m\pi(p^m)}}
=\frac{p^m}{{\pi(p^m)\ln p^m}}\ge \frac{1}{2}.$$
This completes the proof.
\end{proof}

For sufficiently large number $n$, $\mathsf P(n)$ may be rather small, for example $\mathsf P(2^m)=2$ for any $m\in \mathbb N$. However, Lemma \ref{lemma3.2} means that $\mathsf M(n)$ cannot be too small for sufficiently large $n$. We shall use Lemma \ref{lemma3.2} to prove our main results in the next section. In the following, we will prove (\ref{eq2}) which shows that actually $\ln n$ is the minimal order of $\mathsf M(n)$.

Let $p_k$ denote the $k$-th prime and $n_k=p_1\cdots p_k\in \mathbb N$. Clearly, we have $\mathsf M(n_k)=p_k$ and $\ln n_k=\vartheta(p_k)$. By the Prime Number Theorem, for any $\epsilon>0$ there exists $k_0(\epsilon)>0$ such that for all $k>k_0(\epsilon)$ we have
$$\frac{\mathsf M(n_k)}{\ln n_k}=\frac{p_k}{\vartheta(p_k)}\le (1+\epsilon).$$
Therefore we have $\liminf\limits_{n\rightarrow\infty}\frac{\mathsf M(n)}{\ln n}\le1$.

Similarly, from Lemma \ref{lemma3.2}, together with the Prime Number Theorem, for any $\epsilon>0$ there exists $n_0(\epsilon)>0$ such that for all $n>n_0(\epsilon)$ we have
$$\frac{\mathsf M(n)}{\ln n}\ge\frac{p^m}{{\pi(p^m)\ln p^m}}\ge (1-\epsilon).$$
Therefore we have $\liminf\limits_{n\rightarrow\infty}\frac{\mathsf M(n)}{\ln n}\ge1$ and
$$\liminf\limits_{n\rightarrow\infty}\frac{\mathsf M(n)}{\ln n}=1.$$
This completes the proof of (\ref{eq2}).

\section{Proof of the main results}

In this section, we shall prove our main results, Theorem \ref{theorem1}. Firstly, we have to verify Conjecture \ref{conj1} for some small primes which are the remaining cases in \cite{GaoThang,HZ,K}.

\begin{lemma}\label{lemma4.1}
For any $n\in\mathbb N$, we have
\begin{enumerate}
\item $\mathsf s_{k2^n}(C_{2^n}^3)=(k+3)2^n-3$, holds for $k\ge 4$;
\item $\mathsf s_{k3^n}(C_{3^n}^3)=(k+3)3^n-3$, holds for $k\ge 6$;
\item $\mathsf s_{k2^n}(C_{2^n}^4)=(k+4)2^n-4$, holds for $k\ge 12$;
\item $\mathsf s_{k3^n}(C_{3^n}^4)=(k+4)3^n-4$, holds for $k\ge 18$.
\end{enumerate}
\end{lemma}
\begin{proof}
By Lemma \ref{lower}, it suffices to prove $\mathsf s_{k\exp(G)}\le k\exp(G)+\mathsf D(G)-1$.

(1) We prove by induction on $k$. Since $2^{2+n}=4\cdot 2^n\ge\mathsf D(C_{2^n}^3)=3\cdot2^n-2$, by Lemma \ref{lemma1.2} we have
$\mathsf s_{4\cdot 2^n}(C_{2^n}^3)=7\cdot2^n-3$. This proves the case $k=4$. Suppose that $k\ge 5$ and the result holds for all positive integers $n$ with $4\le n\le k$. Now we need to prove $\mathsf s_{(k+1)2^n}(C_{2^n}^3)=(k+1+3)2^n-3$.

Let $S$ be any sequence over $C_{2^n}^3$ of length $$|S|=(k+1+3)2^n-3.$$
By Lemma \ref{lemma1.3}.(1) and the fact that $|S|\ge 8\cdot2^n-7$, we have $S$ contains a zero-sum subsequence $T$ of length $|T|=2^n$. Since
$$|S\bdot T^{-1}|=(k+3)2^n-3=\mathsf s_{k 2^n}(C_{2^n}^3),$$ we have $S\bdot T^{-1}$ contains a zero-sum subsequence $U$ of length $|U|=k2^n$. Consequently, $T\bdot U$ is a zero-sum subsequence of $S$ of length $|T\bdot U|=(k+1)2^n$. This completes the proof.

(2) We prove by induction on $k$. Since $3^{1+n}=3\cdot 3^n\ge\mathsf D(C_{3^n}^3)=3\cdot3^n-2$, by Lemma \ref{lemma1.2} we have
$\mathsf s_{6\cdot 3^n}(C_{3^n}^3)=9\cdot3^n-3$. This proves the case $k=6$. Suppose that $k\ge 7$ and the result holds for all positive integers $n$ with $6\le n\le k$. Now we need to prove $\mathsf s_{(k+1)3^n}(C_{3^n}^3)=(k+1+3)3^n-3$.

Let $S$ be any sequence over $C_{3^n}^3$ of length $$|S|=(k+1+3)3^n-3.$$
By Lemma \ref{lemma1.3}.(2) and the fact that $|S|\ge 9\cdot3^n-8$, we have $S$ contains a zero-sum subsequence $T$ of length $|T|=3^n$. Since
$$|S\bdot T^{-1}|=(k+3)3^n-3=\mathsf s_{k 3^n}(C_{3^n}^3),$$ we have $S\bdot T^{-1}$ contains a zero-sum subsequence $U$ of length $|U|=k3^n$. Consequently, $T\bdot U$ is a zero-sum subsequence of $S$ of length $|T\bdot U|=(k+1)3^n$. This completes the proof.

(3) We prove by induction on $k$. Since $2^{2+n}=4\cdot 2^n\ge\mathsf D(C_{2^n}^4)=4\cdot2^n-3$, by Lemma \ref{lemma1.2} we have
$\mathsf s_{12\cdot 2^n}(C_{2^n}^4)=16\cdot2^n-4$. This proves the case $k=12$. Suppose that $k\ge 13$ and the result holds for all positive integers $n$ with $12\le n\le k$. Now we need to prove $\mathsf s_{(k+1)2^n}(C_{2^n}^4)=(k+1+4)2^n-4$.

Let $S$ be any sequence over $C_{2^n}^4$ of length $$|S|=(k+1+4)2^n-4.$$
By Lemma \ref{lemma1.3}.(3) and the fact that $|S|\ge 16\cdot2^n-15$, we have $S$ contains a zero-sum subsequence $T$ of length $|T|=2^n$. Since
$$|S\bdot T^{-1}|=(k+4)2^n-4=\mathsf s_{k 2^n}(C_{2^n}^4),$$ we have $S\bdot T^{-1}$ contains a zero-sum subsequence $U$ of length $|U|=k2^n$. Consequently, $T\bdot U$ is a zero-sum subsequence of $S$ of length $|T\bdot U|=(k+1)2^n$. This completes the proof.

(4) We prove by induction on $k$. Since $3^{2+n}=9\cdot 3^n\ge\mathsf D(C_{3^n}^4)=4\cdot3^n-4$, by Lemma \ref{lemma1.2} we have
$\mathsf s_{18\cdot 3^n}(C_{3^n}^4)=22\cdot3^n-4$. This proves the case $k=18$. Suppose that $k\ge 19$ and the result holds for all positive integers $n$ with $18\le n\le k$. Now we need to prove $\mathsf s_{(k+1)3^n}(C_{3^n}^4)=(k+1+4)3^n-4$.

Let $S$ be any sequence over $C_{3^n}^4$ of length $$|S|=(k+1+4)3^n-4.$$
By Lemma \ref{lemma1.3}.(4) and the fact that $|S|\ge 20\cdot3^n-19$, we have $S$ contains a zero-sum subsequence $T$ of length $|T|=3^n$. Since
$$|S\bdot T^{-1}|=(k+4)3^n-4=\mathsf s_{k 3^n}(C_{3^n}^4),$$ we have $S\bdot T^{-1}$ contains a zero-sum subsequence $U$ of length $|U|=k3^n$. Consequently, $T\bdot U$ is a zero-sum subsequence of $S$ of length $|T\bdot U|=(k+1)3^n$. This completes the proof.
\end{proof}

\begin{corollary}\label{coro2}
Let $n,m\in\mathbb N$ and $p$ be any prime. we have
\begin{enumerate}
\item $\mathsf s_{kp^m}(C_{p^m}^3)=(k+3)p^m-3$ holds for $k\ge6$;
\item $\mathsf s_{kp^m}(C_{p^m}^4)=(k+4)p^m-4$ holds for $k\ge18$;
\end{enumerate}
\end{corollary}
\begin{proof} By Lemma \ref{lemma4.1}, the results hold for the cases $p=2,3$. For $p\ge 5$, see Theorem 1.(3) in \cite{K} and Theorem 1.2.(3) in \cite{HZ}.
\end{proof}

The following crucial lemma is based on a standard argument in zero-sum theory (we refer to \cite{GH}, Proposition 5.7.11).
\begin{lemma}\label{lemma4.2}
Let $n,m,p,k,r\in \mathbb N$ and $p$ a prime. Assume that $\mathsf s_{kp^m}(C_{p^m}^r)=(k+r)p^m-r$.
Then we have
$$\mathsf s_{knp^m}(C_{np^m}^r)\le(k+r)np^m+a_rn,$$
where $a_r$ is a constant depends on $r$.
\end{lemma}

\begin{proof}
Let $S$ be a sequence of length $|S|=((k+r)p^m-r)n+\mathsf s(C_n^r)$ over $C_n^r$.
Consider the following map:
$$\varphi:C_{np^m}^r\rightarrow C_n^r.$$
Then $\varphi(S)$ is a sequence over $C_n^r$ of length $((k+r)p^m-r)n+\mathsf s(C_n^r)$. By the definition of $\mathsf s(C_n^r)$, we have $\varphi(S)$ contains at least $(k+r)p^m-r$ zero-sum subsequences $S_1,\ldots,S_{(k+r)p^m-r}$ over $C_n^r$ with $|S_i|=n$ for $1\le i\le (k+r)p^m-r$. This means that $$\sigma(S_1),\ldots,\sigma(S_{(k+r)p^m-r})\in \ker(\varphi)=C_{p^m}^r.$$
By the assumption that $\mathsf s_{kp^m}(C_{p^m}^r)=(k+r)p^m-r$, there exist $$\{i_1,\ldots,i_{kp^m}\}\subset\{1,\ldots,(k+r)p^m-r\}$$ such that $\sigma(S_{i_1})+\ldots+\sigma(S_{i_{kp^m}})=0$ and this implies that $S_{i_1}\bdot\ldots\bdot S_{i_{kp^m}}$ is a zero-sum subsequence of $S$ over $C_{np^m}^r$ of length $knp^m$. Therefore
$$\mathsf s_{knp^m}(C_{np^m}^r)\le ((k+r)p^m-r)n+\mathsf s(C_n^r).$$
Moreover, by Lemma \ref{lemma2.1}, there exists an absolute constant $c$ such that
$$((k+r)p^m-r)n+\mathsf s(C_n^r)\le ((k+r)p^m-r)n+(cr\log_2r)^rn.$$
Let $a_r=(cr\log_2r)^r-r$, then we have the desired result.
\end{proof}
For any fixed $r\in\mathbb N$, we denote $a_r=(cr\log_2r)^r-r$, where $c$ is the absolute constant mentioned in Lemma \ref{lemma2.1}. The following corollary is an easy consequence of the above lemma.
\begin{corollary}\label{coro3}
Let $n,k,r\in \mathbb N$. Assume that $\mathsf M(n)=p^m$ and
$$\mathsf s_{kp^m}(C_{p^m}^r)=(k+r)p^m-r.$$
Then we have
$$\mathsf s_{kn}(C_n^r)\le (k+r)n+a_r\frac{n}{\mathsf M(n)}.$$
\end{corollary}

By Corollary \ref{coro3}, in order to prove the main results, it suffices to combine the results about $\mathsf M(n)$ in Section 2.
\medskip

{\sl Proof of the Theorem \ref{theorem1}.}
(1) By Corollary \ref{coro2}.(1) and \ref{coro3}, for $k\ge 6$, we have
$$\mathsf s_{kn}(C_n^3)\le(k+3)n+a_3\frac{n}{\mathsf M(n)}.$$
By lemma \ref{lemma3.2}, for $k\ge 6$, actually we have
$$\mathsf s_{kn}(C_n^3)\le(k+3)n+2a_3\frac{n}{\ln n}$$
and we get the desired result.

(2) By Corollary \ref{coro2}.(2) and \ref{coro3}, for $k\ge 18$, we have
$$\mathsf s_{kn}(C_n^4)\le(k+4)n+a_4\frac{n}{\mathsf M(n)}.$$
By lemma \ref{lemma3.2}, for $k\ge 18$, actually we have
$$\mathsf s_{kn}(C_n^4)\le(k+4)n+2a_4\frac{n}{\ln n}$$
and we get the desired result.

(3) By Lemma \ref{lemma1.2} and Corollary \ref{coro3}, for any $k\in \mathbb N$, we have
$$\mathsf s_{k\mathsf p(n,r)n}(C_n^r)\le(k\mathsf p(n,r)+r)n+a_r\frac{n}{\mathsf M(n)}.$$
By lemma \ref{lemma3.2}, for any $k\in \mathbb N$, actually we have
$$\mathsf s_{k\mathsf p(n,r)n}(C_n^r)\le(k\mathsf p(n,r)+r)n+2a_r\frac{n}{\ln n}$$
and we get the desired result.\qed
\bigskip

\section{Further studies about $\mathsf M(n)$ and some improvements}

In this section, we will provide some further estimates for $\mathsf M(n)$ in some special cases. With these further estimates, we can improve our main results. All these results can be seen as some applications of $\mathsf M(n)$.

Note that, by (\ref{eq2}), it is impossible to improve the order of the lower bound in Lemma \ref{lemma3.2} of $\mathsf M(n)$ for every $n$ any more. However, we can get some better estimates in some special cases.

We denote
$$\mathsf E(x,y)=\{n\le x\text{ }|\text{ }\mathsf M(n)\le y\}$$
and $\overline{\mathsf E(x,y)}=\{n\le x\text{ }|\text{ }n\notin \mathsf E(x,y)\}$.
Let $A>1$ be any real number, in the following we shall consider
$$\mathsf E(x,(\ln x)^A)=\{n\le x\text{ }|\text{ }\mathsf M(n)\le (\ln x)^A\},$$ where $A\ge 1$.
Actually, we have the following lemma.

\begin{lemma}\label{lemma3.3} For any positive integer $A\ge 1$ and $\epsilon>0$, we have
$$|\mathsf E(x,(\ln x)^A)|=O_{A,\epsilon}(x^{1-\frac{1}{A}+\epsilon}).$$
\end{lemma}
\begin{proof} Clearly we have $|\mathsf E(x,(\ln x)^A)|=\sum\limits_{n\le x\atop \mathsf M(n)\le (\ln x)^A}1$, then for any $\delta>0$, we have
$$\sum\limits_{n\le x\atop \mathsf M(n)\le (\ln x)^A}1\le
\sum\limits_{n\le x\atop \mathsf M(n)\le (\ln x)^A}\big(\frac{x}{n}\big)^\delta.$$
Similar to the Euler product of the Riemann zeta function, by the fundamental theorem of arithmetic, we have
\begin{align*}
\sum\limits_{n\le x\atop \mathsf M(n)\le (\ln x)^A}\big(\frac{x}{n}\big)^\delta
&\le x^\delta\prod_{p\le (\ln x)^A}(1-\frac{1}{p^{\delta}})^{-1}\\
&= x^\delta\prod_{p\le (\ln x)^A}(1+\frac{1}{p^{\delta}-1}).
\end{align*}
If we take $c_{\delta}=\frac{2^{\delta}}{2^{\delta}-1}$, then $\frac{1}{p^{\delta}-1}\le \frac{c_{\delta}}{p^{\delta}}$ and we have
$$x^\delta\prod_{p\le (\ln x)^A}(1+\frac{1}{p^{\delta}-1})\le x^\delta\prod_{p\le (\ln x)^A}(1+\frac{c_{\delta}}{p^{\delta}}).$$
As $1+x\le e^x$ for any $x\ge 0$, we have
\begin{align*}
x^\delta\prod_{p\le (\ln x)^A}(1+\frac{c_{\delta}}{p^{\delta}})&\le x^\delta\prod_{p\le (\ln x)^A}\exp(\frac{c_{\delta}}{p^{\delta}})\\
&=x^\delta \exp(\sum_{p\le (\ln x)^A}\frac{c_{\delta}}{p^{\delta}}).
\end{align*}
To estimate the last sum, we employ the relation between the sum and the integral,
\begin{align*}
x^\delta \exp(\sum_{p\le (\ln x)^A}\frac{c_{\delta}}{p^{\delta}})&\le x^\delta \exp(\sum_{2\le n\le (\ln x)^A}\frac{c_{\delta}}{n^{\delta}})\\
&\le x^\delta \exp(c_{\delta}\int_1^{(\ln x)^A}\frac{1}{t^{\delta}}dt).
\end{align*}
Therefore
\begin{align*}
x^\delta \exp(c_{\delta}\int_1^{(\ln x)^A}\frac{1}{t^{\delta}}dt)&= x^\delta \exp\big(\frac{c_{\delta}}{1-\delta}((\ln x)^{A(1-\delta)}-1)\big)\\
&=\exp(\frac{c_{\delta}}{\delta-1})x^\delta \exp\big(\frac{c_{\delta}}{1-\delta}((\ln x)^{A(1-\delta)})\big).
\end{align*}
Now, we take $\delta=1-\frac{1}{A}+\frac{\epsilon}{2}$. Since $$A(1-\delta)=1-\frac{A\epsilon}{2}<1$$ and $$\exp(\frac{c_{\delta}}{1-\delta}((\ln x)^{A(1-\delta)}))=O_{A,\epsilon}(x^{\frac{\epsilon}{2}}),$$ we have
$$\exp(\frac{c_{\delta}}{\delta-1})x^\delta \exp(\frac{c_{\delta}}{1-\delta}((\ln x)^{A(1-\delta)}))=O_{A,\epsilon}(x^{1-\frac{1}{A}+\epsilon}).$$ This completes the proof.
\end{proof}

Recall that for any fixed $r\in\mathbb N$, we denote $a_r=(cr\log_2r)^r-r$, where $c$ is the absolute constant mentioned in Lemma \ref{lemma2.1}. Let
\[
 \mathbb S_{k}^r(x,A)= \left\{ \begin{array}{ll}&\{n\le x\text{ }|\text{ }\mathsf s_{kn}(C_n^r)\le(k+r)n+a_r\frac{n}{(\ln n)^A}\}, \mbox{ if } r=3 \mbox{ or 4}, \\&\{n\le x\text{ }|\text{ }\mathsf s_{k\mathsf p(n,r)n}(C_n^r)\le(k\mathsf p(n,r)+r)n+a_r\frac{n}{(\ln n)^A}\}, \mbox{ if } r\ge5
\end{array} \right.
\]
and
$$\overline{\mathbb S_{k}^r(x,A)}=\{n\le x\text{ }|\text{ }n\notin \mathbb S_{k}^r(x,A)\}.$$

\begin{theorem}\label{theorem4.1}
For any $A\ge 1$ and $\epsilon>0$, we have
\begin{enumerate}
\item $|\overline{\mathbb S_{k}^3(x,A)}|=O_{A,\epsilon}(x^{1-\frac{1}{A}+\epsilon}),$
holds for $k\ge 6$;
\item $|\overline{\mathbb S_{k}^4(x,A)}|=O_{A,\epsilon}(x^{1-\frac{1}{A}+\epsilon}),$
holds for $k\ge 18$;
\item $|\overline{\mathbb S_{k}^r(x,A)}|=O_{A,\epsilon}(x^{1-\frac{1}{A}+\epsilon}),$
holds for $r\ge 5$ and any $k\in\mathbb N$.
\end{enumerate}
In particular,
\begin{enumerate}
\item For $k\ge 6$, we have $$\lim_{x\rightarrow\infty}\frac{|\overline{\mathbb S_{k}^3(x,A)}|}{x}=0;$$
\item For $k\ge 18$, we have $$\lim_{x\rightarrow\infty}\frac{|\overline{\mathbb S_{k}^4(x,A)}|}{x}=0;$$
\item For $r\ge 5$ and any $k\in\mathbb N$, we have $$\lim_{x\rightarrow\infty}\frac{|\overline{\mathbb S_{k}^r(x,A)}|}{x}=0;$$
\end{enumerate}
\end{theorem}
\begin{proof}
By the definition of $\overline{\mathsf E(x,(\ln x)^A)}$ and Corollary \ref{coro3}, it is easy to see that
$$\overline{\mathsf E(x,(\ln x)^A)}\subset \mathbb S_{k}^3(x,A).$$
Therefore, we have
$$\overline{\mathbb S_{k}^3(x,A)}\subset \mathsf E(x,(\ln x)^A).$$
The desired result follows from Lemma \ref{lemma3.3}. In particular, let $\epsilon=\frac{1}{2A}$ in the above result, we have
$$\lim_{x\rightarrow\infty}\frac{|\overline{\mathbb S_{k}^3(x,A)}|}{x}\le\lim_{x\rightarrow\infty}\frac{|\mathsf E(x,(\ln x)^A)|}{x}
=\lim_{x\rightarrow\infty}\frac{O(x^{1-\frac{1}{2A}})}{x}=0.$$
This completes the proof of (1). The proofs of (2) and (3) are similar.
\end{proof}
\medskip
\begin{remark}
According to Theorem \ref{theorem4.1}, for any $A\ge 1$, roughly speaking, for almost every $n\ge 1$ we have
\begin{enumerate}
\item For $k\ge 6$, $$\mathsf s_{kn}(C_n^3)\le (k+3)n+a_3\frac{n}{(\ln n)^A};$$
\item For $k\ge 18$, $$\mathsf s_{kn}(C_n^4)\le (k+4)n+a_4\frac{n}{(\ln n)^A};$$
\item For any $k\in\mathbb N$ and fixed $r\ge 5$, $$\mathsf s_{k\mathsf p(n,r)n}(C_n^r)\le (k\mathsf p(n,r)+r)n+a_r\frac{n}{(\ln n)^A}.$$
\end{enumerate}
\end{remark}
\medskip

Let $\omega(n)$ denote the number of distinct prime divisors of $n$. In the following, we can improve the error terms for some $n\in\mathbb N$ when $\omega(n)$ is a given integer $m$.

\begin{lemma}\label{lemma4.2}
For any $n\in\mathbb N$, we have
$$\mathsf M(n)\ge n^{\frac{1}{\omega(n)}}.$$
\end{lemma}
\begin{proof}
For any $n\in\mathbb N$, we assume that $n=q_1^{r_1}\cdots q_m^{r_m},$
where $q_1^{r_1}<\cdots< q_m^{r_m}$ and $q_1,\ldots,q_m$ are distinct prime numbers, $r_1,\ldots,r_m\in\mathbb N$. Then by the definition of $\mathsf M(n)$, we have $\mathsf M(n)=q_m^{r_m}$.

Clearly we have $\omega(n)=m$ and
$$n=q_1^{r_1}\cdots q_m^{r_m}\le q_m^{mr_m}=\mathsf M(n)^{\omega(n)},$$
consequently $\mathsf M(n)\ge n^{\frac{1}{\omega(n)}}$.
\end{proof}

\begin{theorem}
Let $n,r,m\in \mathbb N$ with $\omega(n)=m$, we have
\begin{enumerate}
\item For $k\ge 6$, $$\mathsf s_{kn}(C_n^3)\le(k+3)n+a_3n^{1-\frac{1}{m}};$$
\item For $k\ge 18$, $$\mathsf s_{kn}(C_n^4)\le(k+4)n+a_4n^{1-\frac{1}{m}};$$
\item For any $k\in\mathbb N$ and fixed $r\ge 5$, 
$$\mathsf s_{k\mathsf p(n,r)n}(C_n^r)\le(k\mathsf p(n,r)+r)n+a_rn^{1-\frac{1}{m}}.$$
\end{enumerate}
\end{theorem}
\begin{proof}
(1) By Corollary \ref{coro2} and \ref{coro3}, Lemma \ref{lemma4.2} and $\omega(n)=m$, for $k\ge 6$ we have
\begin{align*}
&\mathsf s_{kn}(C_n^3)\le(k+3)n+a_3\frac{n}{\mathsf M(n)}=(k+3)n+a_3\frac{n}{n^{\frac{1}{\omega(n)}}}\\
&=(k+3)n+a_3\frac{n}{n^{\frac{1}{m}}}=(k+3)n+a_3n^{1-\frac{1}{m}}.
\end{align*}
This completes the proof.
The proofs of (2) and (3) are similar.
\end{proof}

\begin{remark}
Compared with the previous error terms $\frac{n}{\ln n}$ and $\frac{n}{(\ln n)^A}$, the error term $n^{1-\frac{1}{m}}$ is a large improvement and it is valid for every $n\in\mathbb N$ with $\omega(n)=m$.
\end{remark}

\subsection*{Acknowledgments}
D.C. Han was supported by the National Science Foundation of China Grant No.11601448 and the Fundamental Research Funds for the Central Universities Grant No.2682016CX121. H.B. Zhang was supported by the National Science Foundation of China Grant No.11671218 and China Postdoctoral Science Foundation Grant No. 2017M620936. The authors would like to thank Prof. Weidong Gao for many useful comments and corrections.

\bibliographystyle{amsplain}

\end{document}